\documentclass[11pt]{amsart}

\usepackage{amsmath,amssymb,amscd,amsthm,graphicx,enumerate}

\usepackage{xcolor}
\usepackage{hyperref}
\hypersetup{breaklinks=true}

\numberwithin{equation}{section}
\theoremstyle{plain}

\makeatletter
\newtheorem*{rep@theorem}{\rep@title}
\newcommand{\newreptheorem}[2]{%
\newenvironment{rep#1}[1]{%
 \def\rep@title{#2 \ref{##1}}%
 \begin{rep@theorem}}%
 {\end{rep@theorem}}}
\makeatother

\newtheorem{theorem}[equation]{Theorem}
\newreptheorem{theorem}{Theorem}

\newtheorem{proposition}[equation]{Proposition}
\newtheorem{lemma}[equation]{Lemma}
\newtheorem{corollary}[equation]{Corollary}

\theoremstyle{remark}

\theoremstyle{definition}

\newcommand{\eps}{\varepsilon}

\begin{document}

\title[Ancient flows with fast or slow convergence]{Ancient cylindrical flows with fast or slow convergence}
\author{Kyeongsu Choi}
\author{Robert Haslhofer}

\begin{abstract}
In a recent breakthrough, Bamler-Lai proved the mean-convex neighborhood conjecture in all dimensions by showing that any nontrivial ancient asymptotically cylindrical mean curvature flow is -- up to splitting Euclidean factors -- either a translating bowl, or an ancient oval, or a translating oval-bowl. In this paper, we provide a short alternative argument for two of the three scenarios. Specifically, using more elementary/traditional methods, we show that if the convergence to the round cylinder is fast then the solution is a bowl times a Euclidean factor, and if the convergence is slow then the solution is an ancient oval. Moreover, the present paper also yields a new proof of the mean-convex neighborhood conjecture for neck-singularities that substantially simplifies and streamlines the approach from our prior work joint with Hershkovits (Acta '22) and Hershkovits-White (Inventiones '22). 
\end{abstract}

\maketitle

\section{Introduction}

In a recent breakthrough \cite{BL1,BL2}, Bamler-Lai classified all ancient asymptotically cylindrical flows, i.e. all ancient unit-regular integral Brakke flows whose tangent flow at $-\infty$ is a cylinder. Specifically, they proved:

\begin{theorem}[Bamler-Lai]
Any nontrivial ancient asymptotically cylindrical flow in arbitrary dimensions is up to splitting Euclidean factors either a translating bowl, or an ancient oval, or a translating oval-bowl.
\end{theorem}

As a corollary, this yielded the solution of Ilmanen's mean-convex neighborhood conjecture in arbitrary dimensions, which in turn implied uniqueness of mean curvature flow through cylindrical singularities, c.f. \cite{HershkovitsWhite}. In their proof Bamler-Lai introduced several brilliant new ideas, most notably their PDE-ODI principle, local Harnack inequality and induction over thresholds method, which will likely also have high impact in other areas.

The result by Bamler-Lai completed and unified a long line of research in the classification of singularities, initiated by Angenent-Daskalopoulos-Sesum \cite{ADS1,ADS2} and Brendle and the first author \cite{BC,BC2}, followed by the proof of the mean-convex neighborhood conjecture for neck-singularities in our work with Hershkovits and White \cite{CHH,CHHW},
and a comprehensive research program in $\mathbb{R}^4$ with our collaborators in \cite{CHH_wing,CHH_translator,CHH_linearized,DH_ovals,DH_shape,DH_no_rotation,CDDHS}, which culminated in our classification of ancient noncollapsed flows in $\mathbb{R}^4$ \cite{CH_ancient_r4}; see also the work of Du and his collaborators \cite{DH_blowdown,DuZhu,CDZ,CDZhao} for $\mathbb{R}^{n+1}$.

The purpose of the present paper is to provide a short alternative argument for two of the three scenarios in the Bamler-Lai classification. To describe our setting, given any nontrivial ancient asymptotically cylindrical flow $\mathcal{M}=\{ M_t \}$ in $\mathbb{R}^{n+1}$, consider the cylindrical function $u=u(y,\omega,\tau)$ that measures the deviation of the renormalized flow $\bar{M}_\tau=e^{\tau/2}M_{-e^{-\tau}}$ from the round cylinder $\Gamma=\mathbb{R}^k\times S^{n-k}(\sqrt{2(n-k)})$. It has been shown by Du-Zhu \cite{DuZhu}, generalizing earlier work by Du and the second author \cite{DH_shape,DH_no_rotation}, that for $\tau\to -\infty$ in suitable coordinates one always has 
\begin{equation}\label{spect_quant_intro}
u=-\sqrt{2(n-k)} \sum_{i\in \mathcal{I}} \frac{y_i^2-2}{4|\tau|}+o(|\tau|^{-1}) \quad \textrm{ for some } \mathcal{I}\subseteq\{ 1,\ldots, k\}.
\end{equation}
Following \cite{CH_ancient_r4}, the three cases $\mathcal{I}=\emptyset$, $|\mathcal{I}|=k$ and $0<|\mathcal{I}|<k$ are refereed to as fast convergence, slow convergence and mixed convergence, respectively. In the first case one in fact has exponential decay, specifically $u=O(e^{\tau/2})$. The main result of the present paper is the following:

\begin{theorem}[fast or slow convergence]\label{thm_fast_or_slow} Let $\mathcal{M}$ be a nontrivial ancient asymptotically cylindrical flow. If the convergence is fast, then $\mathcal{M}$ is $\mathbb{R}^{k-1}\times\mathrm{Bowl}_{n+1-k}$, and if the convergence is slow, then $\mathcal{M}$ is an ancient oval.
\end{theorem}

Let us briefly outline our approach. For the case of fast convergence we use comparison with a best fitting translator.\footnote{This idea has been introduced first in \cite{CHH_translator}, where the parameter was the spectral eccentricity, and then implemented in full generality by Bamler-Lai \cite{BL1,BL2}.} Specifically, in the case of fast convergence the renormalized flow centered at any space-time point $X=(x,t)$ with $x=(x_1,\ldots,x_k,0,\ldots,0)$ has the fine cylindrical expansion
 \begin{equation}\label{eq_fine_neck}
 \hat{u}^X= \sum_{i=1}^k a_i y_i e^{\tau/2} + O(e^{\frac{5\tau}{9}})
 \end{equation}
 in Gaussian $L^2$-norm thanks to \cite{DH_blowdown,DuZhu}. Choosing $\mathcal{N}$ to be a translating $\mathbb{R}^{k-1}\times\mathrm{Bowl}_{n+1-k}$ with matching parameters $a_i(\mathcal{N})=a_i(\mathcal{M})$ the difference function $w^{X}=\hat{u}^X-\hat{v}^X$ satisfies $w^X=o(e^{\tau/2})$. Analyzing the evolution of $w^X$, we then show that actually $w^X=O(e^{50\tau})$, provided we shift $\mathcal{N}$ by an optimal parameter $\eta=\eta(\mathcal{M})$. This rapid decay allows us to conclude using the strong maximum principle for Brakke flows from \cite{CHHW}. We note that in the case of fast convergence we can work with the graphical radius $\rho(\tau)=e^{-\tau/10}$, and thus the cutoff errors are super-exponentially small. In particular, in contrast to the sophisticated PDE-ODI principle from \cite{BL1,BL2}, which is applicable in a very general setting, for the case of fast convergence the classical Merle-Zaag lemma from \cite{MZ} is actually enough.
 
In the case of slow convergence we use an elementary geometric argument to show that at any $t_0\ll 0$ the distance $d$ between any two centers of $\eps$-bubble-sheets is bounded by $d\leq |t_0|^{1/2}(\log|t_0|)^2$. Loosely speaking, the idea is that further back in time, specifically for $\tau=-\log(d^2)$, by the expansion \eqref{spect_quant_intro} we see a predictable deviation from the round bubble-sheet, but on the other hand by the Merle-Zaag type ODEs this deviation grows exponentially going forward in time, which would yield a contradiction for too large $d$. Using this bound, and taking into account the classification in the case of fast convergence, it is then easy to infer that any solution with slow convergence is compact and at any $t_0\ll 0$ every point is $\eps$-close to either $\mathbb{R}^k\times S^{n-k}$ or $\mathbb{R}^{k-1}\times\mathrm{Bowl}_{n+1-k}$, which is enough to conclude.
 
Theorem \ref{thm_fast_or_slow} leaves untouched only the case of mixed convergence, which (after splitting off trivial Euclidean factors) occurs if and only if the flow is an oval-bowl. The case of mixed convergence has been resolved first for noncollapsed solutions in $\mathbb{R}^4$ in our paper \cite{CH_ancient_r4}, and then in full generality in $\mathbb{R}^{n+1}$ by Bamler-Lai \cite{BL1,BL2}.
Moreover, since mixed convergence can only occur for bubble-sheets, i.e. for $k\geq 2$, the present paper also yields a new proof of the mean-convex neighborhood conjecture for neck-singularities, i.e. when the tangent flow is a round shrinking $\mathbb{R}\times S^{n-1}$.

\begin{corollary}[neck-singularities]
For the mean curvature flow in arbitrary dimensions all neck-singularities have a mean-convex neighborhood. In particular, mean curvature flow through neck-singularities is unique.
\end{corollary}

Indeed, by Theorem \ref{thm_fast_or_slow} any nontrivial blowup limit near a neck-singularity must be either a translating bowl, ancient oval, or a round shrinking cylinder or sphere, which yields mean-convex neighborhoods by a standard argument (see the final sections of \cite{CHH,CHHW}, and see also \cite{BL2,BaoHaslhofer} for streamlined expositions), and hence yields uniqueness thanks to \cite{HershkovitsWhite}.
To conclude, let us note that the present paper substantially simplifies and streamlines the approach from our prior work joint with Hershkovits and White \cite{CHH,CHHW}, and in particular bypasses a lot of geometric measure theory including the difficult Hopf lemma for Brakke flows.\\

\noindent\textbf{Acknowledgments.}
KC has been supported by the KIAS Individual Grant MG078902, an Asian Young Scientist Fellowship, and the National Research Foundation (NRF) grants RS-2023-00219980 and RS-2024-00345403 funded by the Korea government (MSIT). RH has been supported by the NSERC Discovery grants RGPIN-2016-04331 and RGPIN-2023-04419.

\bigskip

\section{Ancient cylindrical flows with fast convergence}

Throughout this section, $\mathcal{M}=\{ M_t \}$ denotes a nontrivial ancient unit-regular integral Brakke flow in $\mathbb{R}^{n+1}$, whose tangent flow at $-\infty$ is a round shrinking cylinder with fast convergence. In other words, in suitable coordinates, the renormalized flow $\bar{M}_\tau=e^{\tau/2}M_{-e^{-\tau}}$ can be expressed as graph of a nontrivial function $u(\cdot,\tau)$ over domains exhausting the cylinder
$
\Gamma=\mathbb{R}^k\times S^{n-k}(\sqrt{2(n-k)})
$
for $\tau\to -\infty$ and we have $u=O(e^{\tau/2})$.\footnote{This holds uniformly on compact subsets, but also in $\mathcal{H}$-norm, see below.}

Thanks to the fast convergence, by \cite[Section 2.3]{DuZhu} the function
\begin{equation}\label{eq_graph_rad}
\rho(\tau)=e^{-\tau/10}
\end{equation}
is an admissible graphical radius, so in particular we have
\begin{equation}\label{eq_adm_rad}
\| u(\cdot,\tau)\|_{C^4(\Gamma\cap B_{2\rho(\tau)})} \leq C\rho(\tau)^{-2}.
\end{equation}

More generally, given any space-time point $X=(x,t)$, 
we work with the cylindrical function $u^X(\cdot,\tau)$ of the renormalized flow
$
\bar{M}^X_\tau=e^{\tau/2}(M_{t-e^{-\tau}}-x)
$
centered at $X$, and also consider the truncated version
\begin{equation}
\hat{u}^X(y,\tau)=u^X(y,\tau)\chi(|y|/\rho(\tau)),
\end{equation}
where $\chi$ is smooth with $\chi(z)=1$ for $|z|\leq 1/2$ and $\chi(z)=0$ for $|z|\geq 1$. 
As usual, most of our estimates will take place in the Hilbert space
\begin{equation}
\mathcal{H}=L^2(\Gamma,e^{-y^2/4}dV_\Gamma),
\end{equation}
and typically will hold for $\tau\leq \mathcal{T}(Z(X))$. Here, fixing a small constant $\eps_0>0$, the cylindrical scale $Z(X)$ is defined as the smallest $r<\infty$, such that the parabolically dilated flow $\mathcal{D}_r(\mathcal{M}-X)$ is $\eps_0$-close in $C^{\lfloor 1/\eps_0 \rfloor}$ in $B(0,1/\eps_0)\times [-2,-1]$ to the round cylinder with radius $r(t)=\sqrt{2(n-k)|t|}$.

The starting point for our analysis is that there exists a nonvanishing vector $a=a(\mathcal{M})\in\mathbb{R}^k$, such that in suitable coordinates the truncated cylindrical function $\hat{u}^X$ of the renormalized flow $\bar{M}^X_\tau$ centered at any space-time point $X=(x,t)$ with $x=(x_1,\ldots,x_k,0,\ldots, 0)$ has the fine expansion
 \begin{equation}\label{eq_fine_neck}
 \hat{u}^X= \sum_{i=1}^k a_i y_i e^{\tau/2} + O(e^{\frac{5\tau}{9}})
 \end{equation}
in $\mathcal{H}$-norm for all $\tau\leq \mathcal{T}(Z(X))$. This has been stated in \cite[Theorem 6.4]{DH_blowdown} under an additional noncollapsing assumption, but using the outer barriers from \cite[Section 2.2]{DuZhu} the proof goes through in the general case as well. 

\subsection{Merle-Zaag analysis for the difference function}

By a rigid motion and scaling we can assume that \eqref{eq_fine_neck} holds with $a=\sqrt{(n-k)/2}e_k$, namely
 \begin{equation}\label{eq_fine_neck_restated}
 \hat{u}^X= \sqrt{(n-k)/2} y_k e^{\tau/2} + O(e^{\frac{5\tau}{9}}).
 \end{equation}
We would like to compare this with the cylindrical function $\hat{v}^X=\hat{v}^{X,\eta}$ of the renormalized flow $\bar{N}_\tau^X=e^{\tau/2}(N_{t-e^{-\tau}}-x)$, where
\begin{equation}
N_t=\Sigma + (t+\eta) e_k,
\end{equation}
and $\Sigma$ denotes $\mathbb{R}^{k-1}$ times the $(n+1-k)$-dimensional bowl with tip at the origin, and the shift parameter $\eta\in \mathbb{R}$ will be chosen below. Since $N_t$ translates with unit speed in $e_k$-direction, we have the same expansion, i.e.
 \begin{equation}
 \hat{v}^X= \sqrt{(n-k)/2} y_k e^{\tau/2} + O(e^{\frac{5\tau}{9}}).
 \end{equation}
Throughout this subsection, we consider the difference function
\begin{equation}
w^X = \hat{u}^X-\hat{v}^X.
\end{equation}
Note that $w^X=o(e^{\tau/2})$, but our goal is to show that $w^X$ in fact decays much faster provided we choose the best fitting shift parameter $\eta$.

\begin{lemma}[evolution equation]\label{lemma_evol_difference} The difference function evolves by
\begin{equation}\label{eq_evol_difference}
w^X_\tau = \mathcal{L} w^X + E^X,
\end{equation}
where $\mathcal{L}=\Delta_\Gamma -\tfrac12 y \cdot \nabla +1$, and the error term satisfies the estimates
\begin{equation}
|\langle  w^X, E^X\rangle_\mathcal{H}| \leq Ce^{\tau/10}(\|  w^X\|_{\mathcal{H}}^2+ \|\nabla  w^X\|_{\mathcal{H}}^2)+ Ce^{100\tau},
\end{equation}
and
\begin{align}
\| P_+ E^X \|_{\mathcal{H}} \leq Ce^{\tau/10}\| w^X\|_{\mathcal{H}}+ Ce^{100\tau},
\end{align}
where $P_+:\mathcal{H}\to \mathcal{H}_+=\mathrm{span}\{1,y_1,\ldots,y_{n+1}\}$ is the projection to the unstable eigenfunctions of $\mathcal{L}$.
\end{lemma}

\begin{proof}Using \cite[Proposition A.1]{DH_shape} and \eqref{eq_adm_rad}, we see that \eqref{eq_evol_difference} holds with
\begin{equation}
E^X=a_{ij} w^X_{ij}+b_i w^X_i+c w^X +d \cdot 1_{\{\rho/2\leq |y|\leq \rho\}},
\end{equation}
where the coefficient functions (which depend on $u^X$, $v^X$ and $\rho$) satisfy
\begin{equation}
\|a\|_{C^4(B_{2\rho}(0))}+\|b\|_{C^4(B_{2\rho}(0))}+\|c\|_{C^4(B_{2\rho}(0))}+\|d\|_{C^4(B_{2\rho}(0))}\leq C\rho^{-1}.
\end{equation}
Hence, remembering \eqref{eq_graph_rad} and using integration by parts we can estimate
\begin{equation}
|\langle  w^X,a_{ij} w^X_{ij}+b_i w^X_i+c w^X\rangle_\mathcal{H}| \leq Ce^{\tau/10}(\|  w^X\|_{\mathcal{H}}^2+ \|\nabla  w^X\|_{\mathcal{H}}^2),
\end{equation}
and
\begin{equation}
|\langle  w^X, d \cdot 1_{\{\rho/2\leq |y|\leq \rho\}}\rangle_\mathcal{H}| \leq C \int_{\{\rho/2\leq |y|\leq \rho\}}e^{-y^2/4}dV_\Gamma\leq Ce^{100\tau},
\end{equation}
which proves the first estimate. Taking into account that $P_{+}$ projects to a finite dimensional space, and observing also that $|P_+w^X|\leq C(1+|y|)$, the second estimate follows similarly.
\end{proof}

\begin{proposition}[improved decay]
 There is $\omega=\omega(X,\eta)\in \mathbb{R}$, such that for all $\tau\leq\mathcal{T}(Z(X))$ we have
\begin{equation}
\|e^{-\tau} w^X - \omega \|_{\mathcal{H}} \leq Ce^{\tau/10}.
\end{equation}
\end{proposition}

\begin{proof}Recall that the only unstable eigenfunctions of $\mathcal{L}$ are the constant function 1 with eigenvalue 1 and the linear functions $y_1,\ldots,y_{n+1}$ with eigenvalue 1/2. Correspondingly, we can decompose the projection operator $P_+:\mathcal{H}\to \mathcal{H}_+=\mathrm{span}\{1,y_1,\ldots,y_{n+1}\}$ as
\begin{equation}
P_{+}=P_{1}+P_{1/2},
\end{equation}
where $P_\lambda$ denotes the projection to the eigenspace with eigenvalue $\lambda$.
We aim to show that in the eigendecomposition of $w^X$ the term $P_{1/2} w^X$ cannot dominate for $\tau\to -\infty$.
To this end, we consider the function
\begin{equation}
W_0=e^{-\tau}\|P_{1/2} w^X\|_{\mathcal{H}}^2,
\end{equation}
as well as the functions
\begin{equation}
W_+=e^{-\tau}\|P_1 w^X\|_{\mathcal{H}}^2, \qquad W_-=e^{-\tau}\|(1-P_+) w^X\|_{\mathcal{H}}^2.
\end{equation}
Using the second estimate from Lemma \ref{lemma_evol_difference} we see that
\begin{align}
|\tfrac{d}{d\tau}W_0|= 2e^{-\tau}| \langle P_{1/2}w^X,P_{1/2}E^X\rangle_{\mathcal{H}} | \leq Ce^{\tau/10}W+Ce^{99\tau},
\end{align}
where we abbreviated $W=W_0+W_{+}+W_{-}$. Similarly, we obtain
\begin{equation}
|\tfrac{d}{d\tau}W_+- W_+|  \leq Ce^{\tau/10}W+Ce^{99\tau}.
\end{equation}
Moreover, using the first estimate from Lemma \ref{lemma_evol_difference}, and employing the good gradient term from $\langle \mathcal{L}  w^X, w^X\rangle_{\mathcal{H}} = \|  w^X\|_{\mathcal{H}}^2-\| \nabla w^X\|_{\mathcal{H}}^2$, we can estimate 
\begin{align}
\tfrac{d}{d\tau}W &\leq W - (2-Ce^{\tau/10})e^{-\tau}\|\nabla w^X\|_\mathcal{H}^2 +Ce^{\tau/10} W+Ce^{99\tau}\nonumber\\
 &\leq W_+-\tfrac{1}{2(n-k)}W_- +Ce^{\tau/10} W+Ce^{99\tau}.
\end{align}
Combining with the above inequalities yields
\begin{align}
\tfrac{d}{d\tau}W_- \leq-\tfrac{1}{2(n-k)}W_- +Ce^{\tau/10} W+Ce^{99\tau}.
\end{align}
We can thus apply the ODE-lemma from Merle-Zaag \cite{MZ} for the functions $W_-$, $W_0$ and $W_++e^{10\tau}$, which yields that for $\tau\to -\infty$ either
\begin{equation}\label{MZ_neutral_dom}
W_-+W_++e^{10\tau}=o(W_0),
\end{equation}
or
\begin{equation}\label{MZ_unstab_dom}
W_0+W_-\leq Ce^{\tau/10}(W_++e^{10\tau}).
\end{equation}
If \eqref{MZ_neutral_dom} held true, then we would have $W_0>0$ and $|\tfrac{d}{d\tau}\log W_0| \leq Ce^{\tau/10}$.
This would yield $e^{-\tau}\| w^X\|^2_{\mathcal{H}} \to K $ as $\tau\to-\infty$ for some $K\neq 0$, contradicting the fact that $\| w^X\|_{\mathcal{H}}=o(e^{\tau/2})$. Hence, \eqref{MZ_unstab_dom} holds, and we have
\begin{equation}
|\tfrac{d}{d\tau}(e^{-\tau}W_+)| \leq Ce^{\tau/10}(e^{-\tau}W_+)+Ce^{98\tau}.
\end{equation}
Integrating this differential inequality yields the assertion.
\end{proof}

\begin{corollary}[best fitting shift]\label{cor_best_shift}
There is a unique $\eta=\eta(X)\in \mathbb{R}$, such that for all $\tau\leq\mathcal{T}(Z(X))$ we have
\begin{equation}\label{eq:superfast_decay}
\|e^{-\tau} w^X\|_{\mathcal{H}} \leq Ce^{\tau/10}.
\end{equation}
\end{corollary}

\begin{proof}
For the bowl soliton the radius as a function of the height satisfies
\begin{equation}
r'(h)=\sqrt{(n-k)/(2h)} +O(h^{-1}).
\end{equation}
This yields
\begin{equation}
\hat{v}^{X,\eta} = \hat{v}^{X,0}-\sqrt{(n-k)/2}e^\tau \eta + O(e^{3\tau/2}),
\end{equation}
and thus implies the assertion.
\end{proof}

\begin{theorem}[rapid decay]\label{thm_rapid_decay}
If $\eta$ is chosen such that \eqref{eq:superfast_decay} holds, then for all $\tau\leq\mathcal{T}(Z(X))$ we have
\begin{equation}
\|  \hat w^X\|_{\mathcal{H}}\leq Ce^{50\tau}.
\end{equation}
\end{theorem}
 
\begin{proof}
Using Lemma \ref{lemma_evol_difference} we see that
\begin{equation}
\frac{d}{d\tau}\| w^X\|_{\mathcal{H}}^2 \leq 2\langle \mathcal{L}  w^X, w^X\rangle_{\mathcal{H}}+Ce^{\tau/10}(\|  w^X\|_{\mathcal{H}}^2+\| \nabla w^X\|_{\mathcal{H}}^2) +Ce^{100\tau}.
\end{equation}
Together with $\langle \mathcal{L}  w^X, w^X\rangle_{\mathcal{H}} = \|  w^X\|_{\mathcal{H}}^2-\| \nabla w^X\|_{\mathcal{H}}^2$ this yields
\begin{equation}
\frac{d}{d\tau}\| w^X\|_{\mathcal{H}}^2 \leq (2+Ce^{\tau/10})\|  w^X\|_{\mathcal{H}}^2+Ce^{100\tau}.
\end{equation}
Since $e^{-2.1\tau}\| w^X\|_{\mathcal{H}}^2\to 0$ as $\tau\to -\infty$ thanks to Corollary \ref{cor_best_shift}, integrating this differential inequality yields the assertion.
\end{proof} 

\begin{corollary}[universal shift parameter]\label{cor_universal_shift}
There exists $\bar \eta\in\mathbb{R}$, such that 
\begin{equation}
\eta(X)=\bar \eta
\end{equation}
for any space-time point $X=(x,t)$ with $x=(x_1,\ldots,x_k,0\ldots, 0)$.
\end{corollary}

\begin{proof}
Note that the assertion clearly holds in the special case when $M_t$ is $\mathbb{R}^{k-1}$ times a bowl. Since by Theorem \ref{thm_rapid_decay} we have rapid decay to this model situation it follows that the assertion holds in the general case as well.
\end{proof}

 \subsection{From rapid decay to uniqueness}
After a shift in $e_k$-direction we can assume that the constant from Corollary \ref{cor_universal_shift} vanishes, namely $\bar \eta=0$, so $\mathcal{N}=\{\Sigma+te_k\}_{t\in\mathbb{R}}$. Then, by Theorem \ref{thm_rapid_decay} the difference function $w^X$ between the truncated cylindrical functions of $\mathcal{M}$ and $\mathcal{N}$ centered at any $X=(x,t)$ with $x=(x_1,\ldots,x_k,0,\ldots,0)$ for all  $\tau\leq\mathcal{T}(Z(X))$ satisfies
\begin{equation}\label{dist_est_transl}
\|   w^X\|_{\mathcal{H}}\leq Ce^{50\tau}.
\end{equation}
To properly exploit this, we need to control the cylindrical scale:

\begin{lemma}[cylindrical scale]\label{lemma_cyl_sc}
There exist $h_*=h_*(\mathcal{M}) < \infty$ and $C=C(\mathcal{M})<\infty$, such that all $X=(x,t)$ with $x=(x_1,\cdots,x_k,0\cdots,0)$, whenever $x_k-t \geq h_*$, then $Z(X)\leq C (x_k-t)^{1/2}$.
\end{lemma}

\begin{proof}
Suppose towards a contradiction there are such $X_i=(x_i,t_i)$ with $h_i=(x_i)_k-t_i\to \infty$ and $Z(X_i)^2/h_i\to \infty$.
Then, by the fine cylindrical expansion \eqref{eq_fine_neck}, the rescaled flows $\mathcal{M}^i=\mathcal{D}_{Z(X_i)^{-1}}(\mathcal{M}-X_i)$ converge to a round shrinking cylinder $\mathcal{M}^\infty$ with axis through the origin. However, since $Z(X_i)\to \infty$ and $\sqrt{h_i}/Z(X_i)\to 0$, our rescaled model solutions $\mathcal{N}^i=\mathcal{D}_{Z(X_i)^{-1}}(\mathcal{N}-X_i)$ converge to a round shrinking cylinder with axis through the origin that becomes extinct at time zero. Hence, by the decay estimate \eqref{dist_est_transl}, the limit shrinking cylinder $\mathcal{M}^\infty$ becomes extinct at time zero as well, which contradicts the fact that it has $Z(0)=1$. This proves the lemma.
\end{proof}

\begin{corollary}[Hausdorff distance]\label{cor_hausdorff_distance}There exists $C_\ast=C_\ast(\mathcal{M})<\infty$, such that the level sets 
$M_t^h=M_t\cap \{x_k=h\}$ and $N_t^h=N_t\cap \{x_k=h\}$ satisfy
\begin{equation}\label{eq:Hausd_decay}
d_{\mathrm{Hausdorff}}(M_t^h,N_t^h) \leq (h-t)^{-20}\,\,\, \textrm{ whenever } h-t\geq C_\ast.
\end{equation}
\end{corollary}

\begin{proof} By combining the decay estimate \eqref{dist_est_transl} and Lemma \ref{lemma_cyl_sc}, we see that some connected component of $M_t^h$ is extremely close to $N_t^h$ at all high enough levels, with the quantitative closeness estimate \eqref{eq:Hausd_decay}. Moreover, our flow $M_t$ has vanishing asymptotic slope in the $\mathbb{R}^{n+1-k}$-directions by \cite[Proposition 2.5]{DuZhu}. Namely, there exists a smooth function $\varphi:\mathbb{R}\to\mathbb{R}_+$ with $\lim_{h\to \pm \infty} \varphi'(h)=0$, such that for all $t$ sufficiently negative we have
\begin{equation}\label{dir_inf}
|x_{k+1}|^2+\ldots+|x_{n+1}|^2\leq |t|\varphi \left( (x_k-t)/\sqrt{|t|} \right)^2.
\end{equation}
Therefore, there are no other connected components.
\end{proof}

\begin{proposition}[cap size]\label{cor_cap_size} We have $\inf_{\mathcal{M}} (x_k -t) > -\infty$.
\end{proposition}

\begin{proof}
By Corollary \ref{cor_hausdorff_distance} for $\tau$ sufficiently negative we have
\begin{equation}\label{dir_fin}
u^{X_0}|_{y_k=-10} \leq -e^{\tau/2}
\end{equation}
uniformly for all centers $X_0=(x_0,t_0)$, where
\begin{equation}
 x_0=((x_0)_1,\ldots,(x_0)_{k-1},(t_0+C_\ast)e_k,0,\ldots,0).
 \end{equation}
 To proceed, recall that Kleene-Moller \cite{KleeneMoller} constructed rotationally symmetric shrinkers $\Sigma_b^d$ in a halfspace that lie outside of the cylinder and have asymptotic slope $b$. Now, consider $\mathbb{R}^{k-1}\times \Sigma_b^{n-k+1}\subset \mathbb{R}^{n+1}$, and observe that this lies outside of $\bar{M}^{X_0}_\tau$ with positive distance for $y_k\to -\infty$ and for $y_k=-10$ thanks to \eqref{dir_inf} and  \eqref{dir_fin}.
Applying the avoidance principle \cite{White_avoidance}, and taking into account that $b>0$ is arbitrary, we thus infer that for $\tau$ sufficiently negative the renormalized flow $\bar{M}^{X_0}_\tau \cap \{y_k\leq -10\}$ lies inside the cylinder $\Gamma$. In terms of the unrescaled flow this mean that there are uniform constants $\Delta<\infty$ and $H<\infty$, such that $M_{t_0-\Delta}\cap \{ x_k- t_0\leq -H\}$ is contained inside a cylinder of radius $\sqrt{2(n-k)\Delta}$. Hence, by comparison (c.f. \cite[Lemma 4.31]{CHH}) we infer that $\inf_{M_t} (x_k-t_0) \geq -C$ for $t\geq t_0+1$.
Since $X_0$ was arbitrary, this proves the proposition.
\end{proof}

\begin{theorem}[fast convergence]\label{thm_fast}
Our ancient cylindrical flow with fast convergence is identical to its best fitting $\mathbb{R}^{k-1}\times\mathrm{Bowl}_{n+1-k}$ translator, namely $\mathcal{M}=\mathcal{N}$.
\end{theorem}

\begin{proof}
Denote by $K_t$ the closed convex domain with $\partial K_t= N_t$. If we move this domain slightly down, then by Corollary \ref{cor_hausdorff_distance} there is no contact at infinity. Namely, for all $\mu>0$ there exists an $h<\infty$ such that $M_t\cap\{ x_k\geq h\} \subset \mathrm{Int}(K_t-\mu e_k)$ for all $t$. Together with Proposition \ref{cor_cap_size} this implies that there is a $\mu<\infty$, such that $M_t\subset K_t-\mu e_k$ for all $t$. Let $\mu_\ast$ be the infimal such $\mu$, and suppose towards a contraction that $\mu_\ast>0$. Then, there exist $X_i=(x_i,t_i)\in\mathcal{M}$, such that $d(x_i,K_{t_i}-\mu_\ast e_k)\to 0$. Take subsequential limits $\bar{\mathcal{M}}$ of $\mathcal{M}-X_i$ and $\bar{\mathcal{K}}$ of $\mathcal{K}-X_i$. Then, $\bar{M_t}\subset \bar{K}_t-\mu_\ast e_k$ for all $t$, and $0\in \bar{M}_0 \cap \bar{K}_0-\mu_\ast e_k$. Hence, the strong maximum principle for integral Brakke flows \cite[Theorem 3.4]{CHHW}, which is applicable since ancient cylindrical flows have entropy less than $2$, yields that $\bar{M}_t=\partial \bar{K}_t - \mu_\ast e_k$. This contradicts the above fact that there is no contact at infinity, and thus shows that $M_t\subset K_t$ for all $t$. Similarly, for upwards shifts by any $\mu>0$ we get that $M_t\cap (K_t+\mu e_k)=\emptyset$ for all $t$. 
\end{proof}

\bigskip

\section{Ancient cylindrical flows with slow convergence}

Throughout this section, $\mathcal{M}=\{ M_t \}$ denotes an ancient cylindrical flow with slow convergence. Namely, the renormalized flow $\bar{M}_\tau=e^{\tau/2}M_{-e^{-\tau}}$ can be expressed as graph of a function $u(\cdot,\tau)$ over domains exhausting the cylinder $\Gamma=\mathbb{R}^k\times S^{n-k}(\sqrt{2(n-k)})$, and for $\tau\to -\infty$ we have
\begin{equation}\label{eq_unst_dom}
\hat{u}=-\sqrt{2(n-k)} \sum_{i=1}^k \frac{y_i^2-2}{4|\tau|} + o(|\tau|^{-1})
\end{equation}
in $\mathcal{H}$-norm. Here, as usual, we work with the truncated graph function
\begin{equation}
\hat{u}^X(y,\omega,\tau)=u^X(y,\omega,\tau)\chi(|y|/\rho(\tau)),
\end{equation}
where $\rho(\tau)$ denotes any admissible graphical radius. In fact, by \cite[Proposition 2.8]{DuZhu} one can take $\rho(\tau)=|\tau+2\log \hat{Z}(X)|^\gamma$ for $\tau\leq \tau_\ast-2\log\hat{Z}(X)$, where here and in the following we abbreviate $\hat{Z}(X)=\max \{ Z(X),1\}$.

The following lemma shows that if at some $\tau_0>-\infty$ the unstable eigenfunctions still account for at least some percentage of the Gaussian $L^2$-norm of $\hat{u}$, then this percentage condition is preserved forward in time.

\begin{lemma}[visible unstable mode]\label{lem:unstable_mode_ODE}For every $\Lambda<\infty$, there exists a $\tau_\Lambda>-\infty$ with the following significance. If $\|\hat u^X\|_{\mathcal{H}} \leq \Lambda \|P_+\hat u^X\|_{\mathcal{H}}$ holds at some $\tau_0\leq \tau_{\Lambda}-2\log\hat{Z}(X)$, then for all $\tau\in [\tau_0,\tau_{\Lambda}-2\log\hat{Z}(X)]$ we have
 \begin{equation}
 \|\hat u^X\|_{\mathcal{H}} \leq \Lambda \|P_+\hat u^X\|_{\mathcal{H}},
 \end{equation}
and thus in particular
\begin{equation}
\tfrac{d}{d\tau} \|P_+\hat u^X\|_{\mathcal{H}} \geq \tfrac{1}{4}\|P_+\hat u^X\|_{\mathcal{H}}.
\end{equation}
\end{lemma}

\begin{proof} Consider the function
\begin{equation}
f(\tau)=\Lambda^2\|P_+\hat u^X\|_{\mathcal{H}}^2-\|\hat u^X\|_{\mathcal{H}}^2.
\end{equation}
Recall from \cite[Proposition 2.1]{DuZhu} that $\hat{u}^X$ evolves according to
\begin{equation}
\| (\partial_\tau-\mathcal{L})\hat{u}^X \|_{\mathcal{H}}\leq C \rho^{-1} \| \hat{u}^X \|_{\mathcal{H}}.
\end{equation}
Consequently, the function $f$ satisfies
\begin{equation}
\tfrac{d}{d\tau} f\geq  (\Lambda^2-1)\|P_+\hat u^X\|_{\mathcal{H}}^2+o(1)\|\hat u^X\|_{\mathcal{H}}^2.
\end{equation}
Hence, fixing $\tau_\Lambda>-\infty$ sufficiently negative, we conclude that if $f(\tau_0)\geq 0$ for some $\tau_0\leq \tau_\Lambda-2\log\hat{Z}(X)$, then $f(\tau)\geq 0$ for all $\tau\in [\tau_0,\tau_{\Lambda}-2\log\hat{Z}(X)]$, as desired.
\end{proof}

\begin{proposition}[size of bubble-sheet region]\label{thm:cylindrical_diameter}
There exist constants $\eps>0$ and $t_\ast=t_\ast(\mathcal{M})>-\infty$ with the following significance. Suppose that $M_{t_0}$ has an $\varepsilon$-bubble-sheet with center $x_0$ and radius $r(x_0,t_0)$, where $t_0\leq t_\ast$. Then,
\begin{equation}\label{eq:locataion_est}
| x_0| \leq  |t_0|^{\frac{1}{2}} (\log |t_0|)^2,
\end{equation}
and 
\begin{equation}
|t_0|^{\frac{1}{2}}(\log |t_0|)^{-2} \leq r(x_0,t_0) \leq 2(n-k)^{1/2}|t_0|^{\frac{1}{2}}.
\end{equation}
\end{proposition} 

\begin{proof}
By rotating coordinates we may assume without loss of generality that $\ell=\langle x_0,e_1\rangle>0$ and $\langle x_0,e_j\rangle = 0$ for $j=2,\ldots,k$. Moreover, since the assertion holds in the central region by  \eqref{eq_unst_dom}, we may assume that
\begin{equation}\label{negl1}
\ell \geq 10|t_0|^{\frac{1}{2}}.
\end{equation}
Also by \eqref{eq_unst_dom} for $\tau$ sufficiently negative we have $u\leq 0$ for $|y|=10$. Hence, using the rotated KM-barriers from \cite[Corollary 2.4]{DuZhu}, which satisfy the boundary condition at infinity  thanks to the vanishing assymptotic slope from \cite[Proposition 2.5]{DuZhu}, we infer that $\bar{M}_\tau\setminus B_{10}(0)$ is contained inside the cylinder $\Gamma$. In other words, for $t$ sufficiently negative we have
\begin{equation}\label{eq:x34_bound}
\sup_{x\in  M_t\setminus B_{10\sqrt{-t}}(0)}\sum_{j= k+1}^{n+1}\langle x,e_j\rangle^2 \leq 2(n-k)|t|.
\end{equation}
In particular, this already shows that
\begin{equation}\label{negl2}
r(x_0,t_0)^2\leq 4(n-k)|t_0|.
\end{equation}

Next, applying the expansion \eqref{eq_unst_dom} at $\tau=-\log(\ell^2)\ll 0$ we see that the unrescaled hypersurface $M_{-\ell^2}$ for $ |x| \leq 10\ell $ satisfies
\begin{equation}
\left(\sum_{j=k+1}^{n+1}\frac{x_j^2}{\ell^2}\right)^{\frac{1}{2}} =  \sqrt{2(n-k)}\left( 1-\sum_{i=1}^k \frac{(x_i/\ell)^2-2}{8\log\ell}+o(|\log\ell|^{-1})\right).
\end{equation}
We now considering the flow $M_{t}^{X_0'}=M_{t+t_0'}-x_0'$ centered at
\begin{equation}
X_0'=(x_0',t_0'),\qquad \textrm{where }\,\, x_0'=\ell e_1,\,\, t_0' = t_0 + \frac{r(x_0,t_0)^2}{2(n-k)}.
\end{equation}
Then, the hypersurface $M_{-\ell^2-t_0'}^{X_0'}$ for $ |x| \leq 9\ell $ satisfies
\begin{multline}
\left(\sum_{j=k+1}^{n+1}\frac{x_j^2}{\ell^2}\right)^{\frac{1}{2}} =  \sqrt{2(n-k)}\left( 1-\frac{([x_1+\ell]/\ell)^2-2}{8\log\ell}\right.\\
\left. -\sum_{i=2}^k \frac{(x_i/\ell)^2-2}{8\log\ell}+o(|\log\ell|^{-1})\right).
\end{multline}
In other words, the profile function $u^{X_0'}$ of the renormalized mean curvature flow $\bar{M}^{X_0'}_{\tau}=e^{\tau/2} M^{X_0'}_{-e^{-\tau}}$ at $\tau_0=-\log(t_0'+\ell^2)$ for $|y
|\leq 9$ satisfies
\begin{equation}\label{eq_unst_dom_rec}
u^{X_0'}(y,\tau_0)=-\sqrt{2(n-k)} \left(\frac{y_1^2+2y_1-1}{8\log\ell}+\sum_{i=2}^k \frac{y_i^2-2}{8\log\ell} + o(|\log\ell|^{-1})\right).
\end{equation}
Remembering that  $y_1$ is an unstable eigenfunction, and taking also into account \cite[Proposition 4.1]{DH_shape}, we thus infer that
\begin{equation}
\|\hat u^{X_0'}(\cdot,\tau_0)\|_{\mathcal{H}} \leq C_{n,k} (\log \ell)^{-1} \leq C_{n,k}^2 \|P_+\hat u^{X_0'}(\cdot,\tau_0)\|_{\mathcal{H}},
\end{equation}
where $C_{n,k}<\infty$ is a numerical constant.

To proceed, note that $\bar{M}^{X_0'}_{\tau}$ is $\eps$-close to the cylinder $\Gamma$ at time
\begin{equation}
\tau_1=-\log\left(\frac{r(x_0,t_0)^2}{2(n-k)}\right).
\end{equation}
Since $\lim_{\tau\to -\infty}\bar{M}^{X_0'}_{\tau}=\Gamma$, it thus follows from the rigidity case of Huisken's monotonicity formula \cite{Huisken_monotonicity}, that the flow $\bar{M}^{X_0'}_{\tau}$ is $\eps'$-close to $\Gamma$ for all $\tau\leq \tau_1$, where $\eps'\to 0$ as $\eps\to 0$. Choosing $\eps$ small enough, we can thus apply Lemma \ref{lem:unstable_mode_ODE}, which yields that for all $\tau\in [\tau_0,\tau_\Lambda+\tau_1]$ we have
\begin{equation}
\|P_+\hat u^{X_0'}(\cdot,\tau)\|_{\mathcal{H}} \geq C_{n,k}^{-1}(\log \ell)^{-1}e^{(\tau-\tau_0)/4}.
\end{equation}
Moreover, since $\bar{M}^{X_0'}_{\tau}$ is $\eps$-close to the cylinder $\Gamma$ at time $\tau=\tau_1$, we have
\begin{equation}
\| P_+\hat u^{X_0'}(\cdot,\tau_1)\|_{\mathcal{H}} \leq C'_{n,k}\eps.
\end{equation}
We have thus shown that
\begin{equation}
\left(\frac{\ell^2+t_0+\frac{r(x_0,t_0)^2}{2(n-k)}}{\frac{r(x_0,t_0)^2}{2(n-k)}}\right)^{1/4}\leq O(\eps) \log \ell .
\end{equation}
Remembering \eqref{negl1} and \eqref{negl2}, we can recast this in the simpler form
\begin{equation}
\frac{\ell}{(\log \ell)^2}\leq O(\eps)r.
\end{equation}
Using \eqref{negl2} again, this implies the assertion.
\end{proof}

\begin{theorem}[slow convergence]
 Let $\mathcal{M}$ be an ancient unit-regular integral Brakke flow in $\mathbb{R}^{n+1}$, whose tangent flow at $-\infty$ is a round shrinking cylinder with slow convergence. Then, $\mathcal{M}$ is an ancient oval.
\end{theorem}
   
\begin{proof} Fixing $\eps>0$ small enough, for any $t\leq t_\ast$ we consider the set
\begin{equation}
\Omega_t=\left\{ x\in\mathbb{R}^k\, : \, M_t \textrm{ has an $\eps$-bubble sheet with center $(x,0)$}\right\}. 
\end{equation}
Note that $\Omega_t$ is open by definition, and $\overline{\Omega}_t$ is compact by Proposition \ref{thm:cylindrical_diameter}.
Now, given any $t_i\to -\infty$ and $x_i \in \partial \Omega_{t_i}$, we consider the rescaled flows
\begin{equation}
\mathcal{M}^i=\mathcal{D}_{1/r(x_i,t_i)}(\mathcal{M}-(x_i,t_i)),
\end{equation}
where $r(x_i,t_i)\to \infty$ again by Proposition \ref{thm:cylindrical_diameter}.
After passing to a subsequence, we can assume $\mathcal{M}^i$ converges to an ancient unit-regular integral Brakke flow $\mathcal{M}^\infty$, which is a nontrivial ancient cylindrical flow by construction. Since it arises as a blowdown limit of an ancient cylindrical flow with slow convergence, $\mathcal{M}^\infty$ is an ancient cylindrical flow with fast convergence, and thus by Theorem \ref{thm_fast} must be a translating $\mathrm{Bowl}_{n+1-k}\times \mathbb{R}^{k-1}$. This shows that for $t$ sufficiently negative every $(x,t)\in\mathcal{M}$ is $\eps$-close to either a bubble-sheet or a translating $\mathrm{Bowl}_{n+1-k}\times \mathbb{R}^{k-1}$. In particular, $\mathcal{M}$ is smooth, compact, and mean-convex. Since its tangent flow at $-\infty$ is a cylinder, it is noncollapsed as well. This proves the theorem.
 \end{proof}

\bigskip

\bibliography{fast_or_slow}

\bibliographystyle{alpha}

\vspace{5mm}

{\sc School of Mathematics, Korea Institute for Advanced Study, 85 Hoegiro, Dongdaemun-gu, Seoul, 02455, South Korea}\\

{\sc Department of Mathematics, University of Toronto,  40 St George Street, Toronto, ON M5S 2E4, Canada}\\

\emph{E-mail:} choiks@kias.re.kr, roberth@math.toronto.edu

\end{document}